\documentclass[12pt,leqno]{amsart}
\usepackage{amsfonts,amsthm,amsmath,mathtools,amssymb,comment,xcolor}

\theoremstyle{plain}
\newtheorem{thm}{Theorem}[section]
\newtheorem{prop}{Proposition}[section]

\theoremstyle{definition}

\newcommand{\FF}{\mathbb{F}}
\newcommand{\ZZ}{\mathbb{Z}}

\newcommand{\C}{\mathcal{C}}

\newcommand{\JE}{\mathcal{J}}
\newcommand{\Jav}{\mathcal{J}^{av}}
\newcommand{\SSn}{\mathcal{S}_n}
\newcommand{\I}{{\rm{I}}}
\newcommand{\II}{{\rm{I\hspace{-.01em}I}}}
\newcommand{\III}{{\rm{I\hspace{-.01em}I\hspace{-.01em}I}}}
\newcommand{\IV}{{\rm{I\hspace{-.01em}V}}}

\DeclareMathOperator{\comp}{comp}

\begin{document}

\title[Average Complete Joint Weight Enumerators]{Average of Complete Joint Weight Enumerators and Self-dual Codes}

\author[Chakraborty]{Himadri Shekhar Chakraborty*}
\thanks{*Corresponding author}
\address
	{
		(1) Graduate School of Natural Science and Technology\\
		Kanazawa University\\  
		Ishikawa 920-1192, Japan\\
		(2) Department of Mathematics, Shahjalal University of Science and Technology\\ Sylhet-3114, Bangladesh, 
	}
\email{himadri-mat@sust.edu}

\author[Miezaki]{Tsuyoshi Miezaki}
\address
	{
		Faculty of Education, University of the Ryukyus\\ Okinawa  
		903-0213, Japan\\ 
	}
\email{miezaki@edu.u-ryukyu.ac.jp} 

\date{\today}

\begin{abstract}
	In this paper, 
	we give a representation of the average of complete joint weight enumerators 
	of two linear codes of length~$n$ over $\FF_q$ and $\ZZ_k$ 
	in terms of the compositions of $n$ and their distributions in the codes. 
	We also obtain a generalization of the representation for the average of $g$-fold complete joint weight enumerators of codes over $\FF_q$ and $\ZZ_{k}$. Finally, the average of intersection numbers of a pair of Type~$\III$ (resp. Type~$\IV$) codes, and its second moment are found. 
\end{abstract}


\maketitle

\section{Introduction}\label{SecIntroduction}

The notion of the joint weight enumerator of two $\FF_q$-linear codes was introduced in~\cite{MMC1972}. Further, the notion of the $g$-fold complete joint weight enumerator of~$g$ linear codes over $\FF_q$ was given in~\cite{SC2000}. 
In recent few years, 
there has been interest in studying linear codes over the finite rings $\ZZ_{k}$ of 
integers modulo $k$ ($k \geq 2$). 
The concept of the~$g$-fold joint weight enumerator and the~$g$-fold multi-weight enumerator of codes over $\ZZ_{k}$ was 
investigated in~\cite{DHO}. 
Furthermore, the average of joint weight enumerators of two binary codes was investigated in~\cite{Y1989} using the ordinary weight distributions of the codes, 
and also their average intersection number was obtained.
Inspired by the relation pointed out in~\cite{MO2019}
between the complete cycle index and the complete weight enumerator, 
an analogue of the relation between
the average complete joint cycle index and 
the average complete joint weight enumerator of codes 
was given in~\cite{CMOxxxx}. 
Consecutively,
the concept of the average complete joint Jacobi polynomial of codes over $\FF_q$ and $\ZZ_{k}$ was defined in~\cite{CMxxx}, and an analogue of the main theorem in~\cite{Y1989} was given.  
In this paper, 
we define the average complete joint weight enumerator of two linear codes 
over~$\FF_q$ or $\ZZ_k$, and give a generalization of the main theorem in~\cite{Y1989}
for it. 
Moreover,
we extend the idea of the average complete joint weight enumerator 
to the average of~$g$-fold complete joint weight enumerators 
of linear codes over $\FF_q$ or $\ZZ_{k}$. 
We take the average on all permutationally (not monomially) 
equivalent linear codes over~$\FF_q$ or $\ZZ_k$. 
In~\cite{Y1991}, 
the average of intersection numbers of a pair of 
Type~$\I$ or Type~$\II$ codes over~$\FF_2$ 
and their second moments were given. 
As a part of our study, 
we give the average of intersection numbers 
and its second moment of a pair of 
Type~$\III$ codes over~$\FF_3$ as well as Type~$\IV$ codes over~$\FF_4$. 

Throughout this paper, we assume that~$R$ denotes either the finite field $\FF_q$ of order $q$, where $q$ is a prime power or the ring $\ZZ_{k}$ of integers modulo $k$ for some integer $k\geq 2$. 
Let $u = (u_1,u_2,\dots, u_n)$ and $v = (v_1,v_2,\dots,v_n)$ be the elements of $R^n$. 
Then the \emph{inner product} 
of~ $u,v \in R^n$ is given by
\[
	u \cdot v 
	:= 
	u_1v_1 + u_2v_2 + \dots + u_nv_n.
\]
If $u \cdot v = 0$, 
we call $u$ and $v$ \emph{orthogonal}. 
An element
$u \in R^n$ is called \emph{self-orthogonal} if 
$u \cdot u = 0$. 
When we consider a vector space over $\FF_4$, 
it is convenient to consider another inner product given by 
$u \cdot v := u_1\bar{v}_1 + u_2\bar{v}_2 + \dots + u_n\bar{v}_n$, 
where $\bar{a}=a^2$ for $a \in \FF_4$.

An $\FF_q$-linear code of length~$n$ is a vector subspace of 
$\FF_q^n$, 
and a $\ZZ_k$-linear code of length~$n$ is an additive group of $\ZZ_{k}^n$. 
Let $C$ be an $R$-linear code of length~$n$. 
The elements of~$C$ are called \emph{codewords}.
The \emph{dual code} of~$C$ is defined as
\[
	C^\perp 
	:= 
	\{
		v\in R^{n} 
		\mid 
		u \cdot v = 0 
		\text{ for all } 
		u\in C
	\}. 
\]
If $C \subseteq C^\perp$, then $C$ is called \emph{self-orthogonal}, 
and if $C = C^\perp$, then $C$ is called \emph{self-dual}. 
Clearly, if $C$ is self-dual, every codeword $u \in C$ is self-orthogonal.

It is well known that 
the length~$n$ of a self-dual code over~$\FF_q$ is even 
and the dimension is $n/2$. 
A self-dual code~$C$ 
over~$\FF_2$ or~$\FF_4$ 
of length $n\equiv 0 \pmod 2$ 
having even weight 
is called \emph{Type~$\I$} or \emph{Type~$\IV$}, respectively. 
A self-dual code $C$ over~$\FF_2$ 
of length $n\equiv 0\pmod 8$ 
is called \emph{Type~$\II$}
if the weight of each codeword of~$C$ 
is a multiple of~$4$. 
Finally, 
a self-dual code $C$ over~$\FF_3$ 
of length $n\equiv 0\pmod 4$ 
is called \emph{Type~$\III$}
if the weight of each codeword of~$C$ 
is a multiple of~$3$. 

Let the elements of 
$R$ be~$0=\omega_0, \omega_1, \dots, \omega_{|R|-1}$ 
in some fixed order.
Then the \emph{composition} of an element $u \in R^{n}$ is defined as
\[
	\comp(u) 
	:= 
	s(u) 
	:= 
	(s_0(u), s_1(u), \dots, s_{|R|-1}(u)),
\]
where $s_i(u)$ denotes the number of coordinates of $u$ that are equal to $\omega_i$. Obviously
\[
	\sum_{i=0}^{|R|-1} s_i(u) = n.
\]

In general, a \emph{composition} $s$ of $n$ is a vector $s = (s_0,s_1, \dots, s_{|R|-1})$ with non-negative integer components such that
\[
	\sum_{i=0}^{|R|-1} s_i = n.
\]

Let $C$ be an $R$-linear code of length~$n$. 
We denote by $T_s^{C}$ the set of codewords of $C$ with composition $s$ and 
by~$A_s^{C} := |T_s^{C}|$, 
that is, the number of codewords $u \in C$ 
such that~$\comp(u) = s$. Then the \emph{complete weight enumerator} of $C$ is defined as:
\begin{align*}
	\C_{C}(x_0, \dots, x_{|R|-1})
	& :=
	\sum_{u \in C} 
	x_0^{s_0(u)} \dots x_{|R|-1}^{s_{|R|-1}(u)}\\
	& = 
	\sum_{s}
	 A_s^C x_0^{s_0} \dots x_{|R|-1}^{s_{|R|-1}},
\end{align*}
where $x_0, \dots, x_{|R|-1}$ are indeterminates and the sum extends over all compositions $s$ of $n$. Let $K := \{0,1,\dots,|R|-1\}$. Then the complete weight enumerator of $C$ can be written as:
\[
	\C_{C}(x_{i} \text{ with } i \in K) 
	= 
	\sum_{s} A_{s}^C \prod_{i=0}^{|R|-1} x_{i}^{s_i}.
\]

Now let $C$ and $D$ be two $R$-linear codes of length $n$. 
We denote by $\eta(u,v)$ the \emph{bi-composition} of the pair 
$(u,v)$ for $u,v \in R^{n}$
which is a vector with non-negative integer components 
$\eta_{\alpha\beta}(u,v)$ 
defined as
\[
	\eta_{\alpha\beta}(u,v) 
	:= 
	\# \{i \mid (u_i,v_i) = (\alpha,\beta)\},
\]
where $(\alpha,\beta) \in R^2$. Clearly
\[
	\sum_{\alpha,\beta \in R} 
	\eta_{\alpha\beta}(u,v) 
	= n. 
\]
In general, a bi-composition $\eta$ of $n$ is a vector with non-negative integer components $\eta_{\alpha\beta}$ such that
\[
	\sum_{\alpha,\beta \in R} 
	\eta_{\alpha\beta} 
	= n. 
\]
The \emph{complete joint weight enumerator} of $C$ and $D$ is defined as
\begin{align*}
	\C\JE_{C,D}(x_{ij} \text{ with } i,j \in K)
	& :=
	\sum_{u \in C, v \in D}
	\prod_{i,j=0}^{|R|-1}
	x_{ij}^{\eta_{\omega_i\omega_j}(u,v)}\\
	& = 
	\sum_{\eta} 
	A_{\eta}^{C,D} 
	\prod_{i,j = 0}^{|R|-1} x_{ij}^{\eta_{\omega_i\omega_j}},
\end{align*}
where $x_{ij}$ for $i,j \in K$ are the indeterminates and $A_{\eta}^{C,D}$ is the number of pair $(u,v)\in C \times D$ such that $\eta(u,v) = \eta$.

We write $\SSn$ for the symmetric group acting on the set $\{1,2,\dots,n\}$, equipped with the composition of permutations. For any $R$-linear code~$C$, the code $C^\sigma:= \{u^{\sigma} \mid u\in C\}$ for some permutation~$\sigma \in \SSn$ is called \emph{permutationally equivalent} to $C$, where $u^{\sigma} := (u_{\sigma(1)},\dots, u_{\sigma(n)})$. Then the \emph{average complete joint weight enumerator} of $R$-linear codes~$C$ and $D$ is defined as
\[
	\C\Jav_{C,D}(x_{ij} \text{ with } i,j \in K) 
	:= 
	\dfrac{1}{n!}\sum_{\sigma \in \SSn} \C\JE_{C^\sigma,D}(x_{ij}).
\]

This paper is organized as follows. 
In Section~\ref{SecMacWilliams}, 
we give the average version of the MacWilliams identity 
for the complete joint weight enumerators. 
In Sections~\ref{SecProofMainTheorem}, 
we give the main result (Theorem~\ref{ThMainTheorem}) of this paper, 
and in Section~\ref{Secgfold}, 
we obtain a generalization of Theorem~\ref{ThMainTheorem} 
for the average of $g$-fold complete joint weight enumerators of codes 
over~$\FF_q$ and~$\ZZ_k$ (Theorem \ref{thm:gen}). 
In Section~\ref{SecAvInterNum}, 
the average of intersection numbers of a pair of 
Type~$\III$ (resp. Type~$\IV$) codes, 
and its second moment are given
(Theorem~\ref{thm:III} and Theorem~\ref{thm:IV}).

\section{MacWilliams Identity}\label{SecMacWilliams}

The MacWilliams identity 
for $g$-fold complete joint weight enumerators of codes over~$\FF_q$ 
was established in~\cite{SC2000}. 
Further, in~\cite{DHO}, 
the MacWilliams identity 
for $g$-fold joint weight enumerators of codes over~$\ZZ_{k}$ was given. 
In this section, 
we study the MacWilliams type identity 
for the average complete joint enumerators over~$R$. 
At the beginning of this section 
we recall~\cite{DHO,MMC1972} 
to take some fixed character over $R$. 

A \emph{character} $\chi$ of $R$ is a homomorphism from the additive group~$R$ to the multiplicative group of non-zero complex numbers. 

Let $R = \FF_q$, where $q=p^f$ for some prime number $p$. 
Again let $F(x)$ be a primitive irreducible polynomial of degree $f$ over $\FF_p$ 
and let $\lambda$ be a root of $F(x)$. 
Then any element $\alpha \in \FF_q$ has a unique representation as:
\begin{equation}\label{EquAlphaRep}
	\alpha 
	= 
	\alpha_0 
	+ 
	\alpha_1 \lambda 
	+ 
	\alpha_2 \lambda^2 
	+ 
	\dots 
	+ 
	\alpha_{f-1} \lambda^{f-1},
\end{equation}
where $\alpha_i \in \FF_p$, 
and $\chi(\alpha) := \zeta_p^{\alpha_0}$, 
where $\zeta_p$ is the primitive $p$-th root~$e^{2{\pi}i/p}$ of unity, 
and $\alpha_0$ is given by (\ref{EquAlphaRep}).

Again if $R = \ZZ_k$, 
then for $\alpha \in \ZZ_{k}$ 
we defined $\chi$ as $\chi(\alpha) := \zeta_k^{\alpha}$, 
where $\zeta_k$ is the primitive $k$-th root~$e^{2{\pi}i/k}$ of unity. 

We have the MacWilliams identity 
for the complete weight enumerator 
of a code
 $C$ over $R$ as follows. 

\begin{thm}[\cite{DHO, MMC1972}]
	For a code $C$ over $R$ we have
	\[
		\C_{C^\perp}(x_i \text{ with } i \in K) = \dfrac{1}{|C|} T_{R} \cdot \C_C(x_i),
	\]
	where $T_{R} = \left(\chi(\alpha\beta)\right)_{\alpha,\beta \in R}$.
\end{thm}

For a code $C$ over $R$ 
let $\tilde{C}$ be either $C$ or $C^\perp$. Then we define
\[
	\delta(C,\tilde{C}) 
	:= 
	\begin{cases}
		0 & \mbox{if} \quad \tilde{C} = C, \\
		1 & \mbox{if} \quad \tilde{C} = C^\perp.  
	\end{cases}
\]
For $\alpha \in R$, 
we define 
$\gamma(\alpha) := j$ if $\alpha = \omega_j$. 
Though it is obvious for the experts, 
we give a proof of the following Theorem for general readers.

\begin{thm}[MacWilliams Identity]\label{ThMacWilliams}
	Let $C$ and $D$ be two $R$-linear codes of length $n$. 
	Then we have
	\[
		\C\JE_{\tilde{C},\tilde{D}}
		(x_{ij} \text{ with } i,j \in K) 
		= 
		\dfrac{1}{|C|^{\delta(C,\tilde{C})}|D|^{\delta(D,\tilde{D})}} T_{R}^{\delta(C,\tilde{C})} 
		\otimes 
		T_{R}^{\delta(D,\tilde{D})} 
		\C\JE_{C,D}(x_{ij}).
	\]
\end{thm}

\begin{proof}
	It is sufficient to show
	\[
		|D|\C\JE_{{C},{D}^\perp}(x_{ij}) 
		= 
		(I \otimes T_{R}) \C\JE_{C,D}(x_{ij}),
	\]
	where, $\tilde{C} = C$, $\tilde{D} = D^\perp$, 
	and 
	$I$ is the identity matrix. 
	Let
	\[
	\delta_{D^\perp}(v) := 
		\begin{cases}
			1 & \mbox{if} \quad v \in D^\perp, \\
			0 & \mbox{otherwise}.
		\end{cases} 
	\]
	Then we have the following identity
	\[
		\delta_{D^\perp}(v) 
		= 
		\dfrac{1}{|D|} \sum_{d \in D} \chi(d \cdot v).
	\]
	Now
	\begin{align*}
		|D|\C\JE_{{C},{D}^\perp}(x_{ij}) 
			& = |D| \sum_{c \in C} \sum_{d^\prime \in {D}^\perp} \prod_{i,j=0}^{|R|-1} x_{ij}^{\eta_{\omega_i\omega_j}(c,d^\prime)} \\
			& = |D| \sum_{c \in C} \sum_{v \in R^n} \delta_{D^\perp}(v) \prod_{i,j=0}^{|R|-1} x_{ij}^{\eta_{\omega_i\omega_j}(c,v)} \\
			& = \sum_{c \in C} \sum_{v \in R^n} \sum_{d \in D} \chi(d \cdot v) \prod_{i,j=0}^{|R|-1} x_{ij}^{\eta_{\omega_i\omega_j}(c,v)} \\
			& = \sum_{c \in C, d \in D} \sum_{v \in R^n} \chi(d \cdot v) \prod_{i,j=0}^{|R|-1} x_{ij}^{\eta_{\omega_i\omega_j}(c,v)} \\
			& = 
			\sum_{c \in C, d \in D} 
			\sum_{(v_1,\dots,v_n) \in R^n} 
			\chi(d_1v_1 + \dots + d_nv_n) 
			\prod_{1\leq i \leq n} 
			x_{\gamma(c_i) \gamma(v_i)} \\
			& = 
			\sum_{c \in C, d \in D} 
			\prod_{1\leq i \leq n} 
			\sum_{v_i \in R} 
			\chi(d_iv_i) 
			x_{\gamma(c_i) \gamma(v_i)} \\
			& = \sum_{c \in C, d \in D} \prod_{(\alpha,\beta) \in R^2} \left(\sum_{v \in R} \chi(\beta v) x_{\gamma(\alpha) \gamma(v)}\right)^{\eta_{\alpha\beta}(c,d)} \\
			& = \C\JE_{C,D}\left(\sum_{v \in R} \chi(\beta v) x_{\gamma(\alpha) \gamma(v)} \text{ with } (\alpha,\beta) \in R^2\right) \\
			& = (I \otimes T_{R}) \C\JE_{C,D}(x_{ij}).
	\end{align*}
	Hence, the proof is completed.
\end{proof}

Now from the above Theorem~\ref{ThMacWilliams}, 
we have the generalized MacWilliams identity 
for the average complete joint weight enumerator 
of codes $C$ and $D$ as follows:
\[
	\C\Jav_{\tilde{C},\tilde{D}}
	(x_{ij} \text{ with } i,j \in K) 
	= 
	\dfrac{1}{|C|^{\delta(C,\tilde{C})}|D|^{\delta(D,\tilde{D})}} T_{R}^{\delta(C,\tilde{C})} 
	\otimes 
	T_{R}^{\delta(D,\tilde{D})} 
	\C\Jav_{C,D}(x_{ij}).
\]

\section{Main Result}\label{SecProofMainTheorem}

In this section, 
we give the main result of 
this paper 
which is presented in the following theorem.
\begin{thm}[Main Theorem]\label{ThMainTheorem}
	Let $C$ and $D$ be two $R$-linear codes of length $n$, 
	and 
	$r$ and $s$ be the compositions of $n$. 
	Again let $\eta$ be the bi-composition of $n$ such that 
	\begin{align*}
		r & = 
		\left(\sum_{\beta \in R} 
		\eta_{\omega_0\beta},\dots,\sum_{\beta\in R} 
		\eta_{\omega_{|R|-1}\beta}\right),\\
		s & = 
		\left(\sum_{\alpha \in R} 
		\eta_{\alpha\omega_0},\dots,\sum_{\alpha\in R} 
		\eta_{\alpha\omega_{|R|-1}}\right).
	\end{align*}
	Then we have
	\begin{align*}
		\C\Jav_{C,D}
		(x_{ij} & \text{ with } i,j \in K) \\
		& = 
		\sum_{r,s,\eta} 
		A_r^C A_s^D 
		\dfrac{\prod\limits_{i=0}^{|R|-1}
		\dbinom{s_i}{\eta_{\omega_0\omega_i}, \dots, \eta_{\omega_{|R|-1}\omega_i}}}
		{\dbinom{n}{r_0, \dots, r_{|R|-1}}} 
		\prod_{i,j=0}^{|R|-1} 
		x_{ij}^{\eta_{\omega_i\omega_j}},
	\end{align*} 
	where
	\[
		\dbinom{a}{b_0, b_1, \dots, b_{m}} 
		:= 
		\dfrac{a!}{b_0! b_1! \dots b_m!}.
	\]
\end{thm}

\begin{proof}
Let $C$ and $D$ be two $R$-linear codes of length~$n$. 
Then the complete joint weight enumerator of $C$ and $D$ is
\begin{equation}\label{EquDefCJWE}
	\C\JE_{C,D}(x_{ij} \text{ with } i,j \in K) 
	:= 
	\sum_{\eta} A_{\eta}^{C,D} 
	\prod_{i,j = 0}^{|R|-1} 
	x_{ij}^{\eta_{\omega_i\omega_j}},
\end{equation}
where $	\sum_{\alpha,\beta \in R} \eta_{\alpha\beta} = n$. 
Now let us define
\[
	B_{r,s,\eta}^{C,D} 
	:= 
	\#\{(u,v) \in C \times D 
	\mid 
	\comp(u)=r, \comp(v) = s, \eta(u,v) = \eta\}.
\]
Therefore, $A_\eta^{C,D} = B_{r,s,\eta}^{C,D}$,	
where 
\begin{align*}
	r & = \left(\sum_{\beta \in R} \eta_{\omega_0\beta},\dots,\sum_{\beta\in R} \eta_{\omega_{|R|-1}\beta}\right),\\
	s & = \left(\sum_{\alpha \in R} \eta_{\alpha\omega_0},\dots,\sum_{\alpha\in R} \eta_{\alpha\omega_{|R|-1}}\right).
\end{align*}
Hence, we can write from~(\ref{EquDefCJWE})
\begin{equation}
	\C\JE_{C,D}(x_{ij} \text{ with } i,j \in K) 
	:= 
	\sum_{r,s,\eta} 
	B_{r,s,\eta}^{C,D} 
	\prod_{i,j = 0}^{|R|-1} 
	x_{ij}^{\eta_{\omega_i\omega_j}}.
\end{equation}
Now
\begin{align*}
	\sum_{\sigma \in \SSn} 
	B_{r,s,\eta}^{C^\sigma,D} 
	& = 
	\# \{(u,v,\sigma) 
	\in 
	T_r^C \times T_s^D \times \SSn 
	\mid 
	\eta(u^\sigma,v) 
	= \eta \}\\		
	& = 
	\sum_{u \in T_r^C}
	\sum_{v \in T_s^D} 
	\# \{\sigma \in \SSn \mid \eta(u^\sigma,v) = \eta\}.
\end{align*}
It is well known that the order of a subgroup of $\SSn$ which stabilizes $u \in T_r^C$ is $\prod_{i=0}^{|R|-1}r_{i}!$. Therefore,

\begin{align*}
	\sum_{\sigma \in \SSn} B_{r,s,\eta}^{C^\sigma,D} 
	&= \sum_{u \in T_r^C}\sum_{v \in T_s^D} \prod_{i=0}^{|R|-1}r_{i}! 
	\# \{u^{\prime} \in R^{n} 
	\mid 
	\comp(u^{\prime}) = r,\eta(u^{\prime},v) = \eta\}\\		
	&= \sum_{u \in T_r^C}\sum_{v \in T_s^D} \prod_{i=0}^{|R|-1}r_{i}! \prod_{i=0}^{|R|-1}\dfrac{s_i!}{\prod_{j=0}^{|R|-1}\eta_{\omega_j\omega_i}!} \\		
	&= A_r^C A_s^D \prod_{i=0}^{|R|-1}r_{i}! \prod_{i=0}^{|R|-1}\dfrac{s_i!}{\prod_{j=0}^{|R|-1}\eta_{\omega_j\omega_i}!}\\
	&= A_r^C A_s^D n!\dfrac{\prod_{i=0}^{|R|-1}\dfrac{s_i!}{\prod_{j=0}^{|R|-1}\eta_{\omega_j\omega_i}!}}{\dfrac{n!}{\prod_{i=0}^{|R|-1}r_{i}!}}\\
	&= 
	A_r^C A_s^D n!
	\dfrac{\prod_{i=0}^{|R|-1}
		\dbinom{s_i}{\eta_{\omega_0\omega_i}, 
		\dots, 
		\eta_{\omega_{|R|-1}\omega_i}}}
	{\dbinom{n}{r_0, r_1, \dots, r_{|R|-1}}}.
\end{align*}
Now we have
\begin{align*}
	\C\Jav_{C,D}&(x_{ij} \text{ with } i,j \in K)\\ 
	&= \dfrac{1}{n!} \sum_{\sigma\in \SSn} \C\JE_{C^\sigma,D}(x_{ij})\\
	&= \dfrac{1}{n!} \sum_{r,s,\eta} \sum_{\sigma\in \SSn} B_{r,s,\eta}^{C^\sigma,D} \prod_{i,j = 0}^{|R|-1} x_{ij}^{\eta_{\omega_i\omega_j}} \displaybreak\\
	&= \sum_{r,s,\eta} A_r^C A_s^D \dfrac{\prod_{i=0}^{|R|-1}\dbinom{s_i}{\eta_{\omega_0\omega_i}, \dots, \eta_{\omega_{|R|-1}\omega_i}}}{\dbinom{n}{r_0, r_1, \dots, r_{|R|-1}}} \prod_{i,j = 0}^{|R|-1} x_{ij}^{\eta_{\omega_i\omega_j}}.
\end{align*}
This completes the proof.
\end{proof}

\section{Average of $g$-fold Complete Joint Weight Enumerators}\label{Secgfold}

In this section, we give a generalization of the Main Theorem for the average $g$-fold complete joint weight enumerators of codes over $R$.

Let $C_1,C_2, \dots, C_g$ be $R$-linear codes of length $n$.
We denote by $\eta^g(c_1,\dots,c_g)$ the \emph{$g$-fold composition} of $g$-tuple 
$$(c_1, \dots , c_g) \in C_1 \times \dots \times C_g,$$ 
which is a vector with non-negative integer components $\eta_a^g(c_1,\dots,c_g)$ for $a \in R^g$ and defined as:
\[
	\eta_a^g(c_1,\dots,c_g) 
	:= 
	\#\{i \mid (c_{1i}, \dots, c_{gi}) = a\}.
\]
We denote by a $g$-fold composition $\eta^{g}$ of $n$ a vector with non-negative integer components $\eta_a^g$ for $a \in R^g$ such that 
\[
	\sum_{a \in R^g} \eta_a^g = n.
\]

We also denote by $T_{\eta^{g}}^{C_1,\dots,C_g}$ the set of codewords of $C_1 \times \dots \times C_g$ with $g$-fold composition $\eta^{g}$. The \emph{$g$-fold complete joint weight enumerator} is defined as follows:
\begin{align*}
	\C\JE_{C_1,\dots,C_g}
	(x_{\gamma(a)} \text{ with } a \in R^g) 
		& := 
		\sum_{c_1 \in C_1, \dots, c_g \in C_g} 
		\prod_{a \in R^g} 
		x_{\gamma(a)}^{\eta_a^g(c_1,\dots, c_g)} \\
		& = 
		\sum_{\eta^g} A_{\eta^g}^{C_1,\dots,C_g} 
		\prod_{a \in R^g} 
		x_{\gamma(a)}^{\eta_a^g},
\end{align*}
where $x_{\gamma(a)}$ for $a \in R^g$ with $\gamma(a) := (\gamma(a_1), \dots, \gamma(a_g))$ are the indeterminates and $A_{\eta^g}^{C_1,\dots,C_g}$ is the number of $g$-tuples $(c_1,\dots, c_g) \in C_1 \times \dots \times C_g$ such that $\eta^g(c_1,\dots,c_g) = \eta^g$. The \emph{average $g$-fold complete joint weight enumerators} are defined as:
\[
	\C\Jav_{C_1,C_2,\dots,C_g}
	(x_{\gamma(a)} \text{ with } a \in R^g) 
	:= 
	\dfrac{1}{n!}
	\sum_{\sigma \in \SSn} 
	\C\JE_{C_1^\sigma,C_2,\dots,C_g}(x_{\gamma(a)}).
\]

Let $a = (a_1,\dots, a_g) \in R^g$ and $b = (b_1,\dots, b_{g-1}) \in R^{g-1}$. Then we denote 
\begin{align*}
	[a;j] & := (a_1,\dots,a_{j-1},a_{j+1},\dots,a_g) \in R^{g-1},\\
	(z;b) & := (z,b_1,\dots,b_{g-1}) \in R^g \text{ for } z \in R.
\end{align*}

Now we have the following generalization of Theorem~\ref{ThMainTheorem}.

\begin{thm}\label{thm:gen}
	Let $C_1,C_2, \dots, C_g$ be the $R$-linear codes of length $n$ and $s_1,s_2, \dots, s_g$ be the composition of $n$. Again let $\eta^g$ be the $g$-fold composition of $n$ such that
	\[
		s_j = \left(\sum_{a \in R^g}\eta_{a}^g \text{ with } a_j=\omega_{i} \text{ for } i \in K\right) \quad \text{ where } j = 1, 2, \dots, g,
	\]
	and $\eta^{g-1}$ be the $(g-1)$-fold composition of $n$ such that the non-negative integer components $\eta_b^{g-1}$ for $b \in R^{g-1}$ is equal to the sum of $\eta_{a}^g$ over all $a \in R^g$ with $[a;1]=b$, that is,
	\[
		\eta_b^{g-1} 
		=  
		\sum_{a \in R^g} \eta_{a|_{[a;1]=b}}^g.
	\]
	Then we have
	\begin{align*}
		& \C\Jav_{C_1,\dots,C_g} (x_{\gamma(a)} \text{ with } a \in R^g) \\
		& = \sum_{s_1,\eta^{g-1},\eta^g} A_{s_1}^{C_1} A_{\eta^{g-1}}^{C_2,\dots,C_g} \dfrac{\prod\limits_{b \in R^{g-1}}\dbinom{\eta_b^{g-1}}{\eta_{(\omega_0;b)}^g, \dots, \eta_{(\omega_{|R|-1},b)}^g}}{\dbinom{n}{s_{10}, \dots, s_{1|R|-1}}} \prod_{a \in R^g} x_{\gamma(a)}^{\eta_{a}^g},
	\end{align*}
	where 
	\[
		\dbinom{r}{r_0, r_1, \dots, r_{m}} 
		:= 
		\dfrac{r!}{r_0! r_1! \dots r_m!}.
	\]
\end{thm}

\begin{proof}
	Let $C_1,\dots, C_g$ be $R$-linear codes of length $n$. Then by the definition of $g$-fold complete joint weight enumerator of the codes $C_1,\dots,C_g$ we have,
	\begin{equation}\label{EquDefCgfoldJWE}
		\C\JE_{C_1,\dots,C_g}
		(x_{\gamma(a)} \text{ with } a \in R^g) 
		:= 
		\sum_{\eta^g} 
		A_{\eta^g}^{C_1,\dots,C_g} 
		\prod_{a \in R^g} 
		x_{\gamma(a)}^{\eta_{a}^g},
	\end{equation}
	where
	\[
		\sum_{a \in R^g} \eta_{a}^g = n.
	\] 
	Now let us define
	\begin{align*}
		B_{s_1,\eta^{g-1},\eta^g}^{C_1,\dots,C_g} 
		:= 
		\#
		\{
			(c_1,\dots,c_g) \in 
			& 
			C_1 \times\dots\times C_g 
			\mid 
			\comp(c_1)=s_1, \\ 
			& 
			\eta^{g-1}(c_2,\dots,c_g) 
			= 
			\eta^{g-1}, 
			\eta^g(c_1,\dots,c_g) 
			= 
			\eta^g
		\}.
	\end{align*}
	Therefore,
	\[
		A_{\eta^g}^{C_1,\dots,C_g} 
		= 
		B_{s_1,\eta^{g-1},\eta^g}^{C_1,\dots,C_g}.
	\]
	Hence, we can write from~(\ref{EquDefCgfoldJWE})
	\begin{equation}
		\C\JE_{C_1,\dots,C_g}(x_{\gamma(a)} \text{ with } a \in R^g) 
		:= 
		\sum_{s_1,\eta^{g-1},\eta^g} 
		B_{s_1,\eta^{g-1},\eta^g}^{C_1,\dots,C_g} 
		\prod_{a \in R^g} 
		x_{\gamma(a)}^{\eta_a^g}.
	\end{equation}
	Now
	\begin{align*}
		& \sum_{\sigma \in \SSn} B_{s_1,\eta^{g-1},\eta^g}^{C_1^\sigma,C_2,\dots,C_g} \\
		= & \# \{(c_1,\dots,c_g,\sigma) \in T_{s_1}^{C_1} \times\dots\times T_{s_g}^{C_g} \times \SSn \mid \eta^g(c_1^\sigma,c_2,\dots,c_g) = \eta^g \} \\		
		= & \sum_{c_1 \in T_{s_1}^{C_1}}\sum_{(c_2,\dots,c_g) \in T_{\eta^{g-1}}^{C_2,\dots,C_g}} \# \{\sigma \in \SSn \mid \eta^g(c_1^\sigma,c_2,\dots,c_g) = \eta^g \}.
	\end{align*}
	It is well known that the order of a subgroup of $\SSn$ which stabilizes $c_1 \in T_{s_1}^{C_1}$ is $\prod_{i=0}^{|R|-1}s_{1i}!$. Therefore,
	
	\begin{align*}
		\mathclap{\sum_{\sigma \in \SSn} 
		B_{s_1,\eta^{g-1},\eta^g}^{C_1^\sigma,C_2,\dots,C_g}} & \notag \\
		& = 
		\sum_{c_1 \in T_{s_1}^{C_1}}
		\sum_{(c_2,\dots,c_g) \in T_{\eta^{g-1}}^{C_2,\dots,C_g}} \prod_{i=0}^{|R|-1}s_{1i}! \notag \\
		& \quad \quad 
		\# \{c_{1}^{\prime}\in R^{n} 
		\mid 
		\comp(c_{1}^{\prime}) = s_1,  
		\eta^g(c_1^\prime,c_2,\dots,c_g) = \eta^g\}\\
		& = 
		A_{s_1}^{C_1} A_{\eta^{g-1}}^{C_2,\dots,C_g} 
		\prod_{i=0}^{|R|-1}s_{1i}!
		\prod\limits_{b \in R^{g-1}}\dfrac{(\eta_b^{g-1})!}{(\eta_{(\omega_0;b)}^g)! \dots (\eta_{(\omega_{|R|-1};b)}^g)!}.
	\end{align*}
	Now it is easy to complete the proof by following similar arguments stated in the proof of Theorem~\ref{ThMainTheorem}.
\end{proof}

\section{The Average of Intersection Numbers}\label{SecAvInterNum}

The notion of 
the average intersection number was introduced in~\cite{Y1989}
for binary linear codes. We take the same notion for $R$-linear codes~$C$ and~$D$ of length~$n$ and define the \emph{average intersection number} as follows:
\[
	\Delta(C,D) 
	:= 
	\dfrac{1}{n!} 
	\sum_{\sigma \in \SSn} 
	|C \cap D^\sigma|.
\]

Now we have the following result.

\begin{prop}
	Let $C,D$ be two $R$-linear code of length $n$, 
	and $r$ be the composition of $n$. 
	Then we have
	\[
		\Delta(C,D) 
		= 
		\sum_{r} 
		\dfrac{A_{r}^{C} A_{r}^{D}}{\dbinom{n}{r_0,\dots,r_{|R|-1}}}.
	\]
\end{prop}

\begin{proof}
	Let $T_{r}^{C}$ and $T_{r}^{D}$ 
	be the set of all elements of 
	$C$ and $D$, respectively, with the composition 
	$r = (r_0,\dots,r_{|R|-1})$ of $n$. 
	Then we can write
	\begin{align*}
		n! \Delta(C,D) 
			& = \sum_{\sigma \in \SSn} |C \cap D^\sigma|\\
			& = \#\{(u,v,\sigma) \in C \times D \times \SSn \mid u = v^\sigma\}\\
			& = \sum_{r} \sum_{u \in T_r^C} \sum_{v \in T_r^D} \#\{\sigma \in \SSn \mid u = v^\sigma\}\\
			& = \sum_{r} A_r^C A_r^D \prod_{i = 0}^{|R|-1} {r_i}!.	
	\end{align*}
	Hence, this completes the proof.
\end{proof}

Let $C \subseteq \FF_q^n$ for $q = 2,3,4$ be a code. Now for $m = 1,2$ we define 
\[
	\Delta_{J}^{m}(C) 
	:= 
	\dfrac{1}{|J_n|} \sum_{D \in J_n} {|C \cap D|}^m,
\]
where $J_n$ denotes the set of self-dual codes of Type $J$, 
where $J$ stands for 
$\I$, $\II$, $\III$ or $\IV$.
The following results for  
$J = \I$ and $\II$ are presented in~\cite{Y1991}. 

\begin{thm}[\cite{Y1991}]\label{ThDeltaOne}
	Let $C$ be a binary self-dual code of length $n$. Then
	\begin{itemize}
		\item [(i)] $\Delta_{\I}(C) \approx 4$ \quad if $C$ is of Type~$\I$,
		\item [(ii)] $\Delta_{\II}(C)\approx 6$ \quad if $C$ is of Type~$\II$.
	\end{itemize}
\end{thm}

\begin{thm}[\cite{Y1991}]\label{ThDeltaTwo}
	Let $C$ be a binary self-dual code of length $n$. Then
	\begin{itemize}
		\item [(i)] $\Delta_{\I}^{2}(C) \approx 24$ \quad if $C$ is of Type~$\I$,
		\item [(ii)] $\Delta_{\II}^{2}(C)\approx 60$ \quad if $C$ is of Type~$\II$.
	\end{itemize}
\end{thm}

In this section, 
we give the analogous results of the above theorems for 
Type~$\III$ and Type~$\IV$ codes over $\FF_3$ and $\FF_4$ respectively. 
Before presenting our findings, 
we adopt the following mass formulas 
which give the numbers of Type~$\III$ and Type~$\IV$ codes 
over $\FF_3$ and $\FF_4$ respectively.

\begin{thm}[\cite{MOSW1978,VP1968}]
	The following hold:
	\begin{itemize}
		\item [(i)] The number of Type~$\III$ codes over $\FF_3$ of length 
		$n \equiv 0 \pmod 4$ is
		\[
			2\prod_{i=1}^{n/2-1} (3^i + 1).
		\]
		\item [(ii)] The number of Type~$\IV$ codes over $\FF_4$ of length 
		$n \equiv 0 \pmod 2$ is
		\[
			\prod_{i=0}^{n/2-1} (2^{2i+1}+1).
		\]
	\end{itemize}
\end{thm}

Let 
$C^\prime \subseteq \FF_3^n$ 
be a self-orthogonal code of dimension $k$. 
We denote by $N_{n,k}^{\III}$ 
the number of Type~$\III$ codes over $\FF_3$ 
of length $n$ containing $C^\prime$. 
Then from~\cite{BT2019} we have
\[
	N_{n,k}^{\III} = 2\prod_{i=1}^{n/2-k-1} (3^i + 1).
\]
For $k = 1$, 
we get from~\cite{PP1973} 
the number of Type~$\III$ codes over $\FF_3$ of length~$n$ 
containing a self-orthogonal vector of $\FF_3^n$.

Now if 
$C^\prime \subseteq \FF_4^n$ 
is a self-orthogonal code having dimension $k$, 
then the number of Type~$\IV$ codes over $\FF_4$ of length $n$ 
containing $C^\prime$, 
denoted by $N_{n,k}^{\IV}$, 
is given in~\cite{CPS1979} as follows: 
\[
	N_{n,k}^{\IV} 
	= 
	\prod_{i=0}^{n/2-k-1} (2^{2i+1}+1).
\]
In particular, for $k = 1$ we get the number from~\cite{MOSW1978}.

The following theorem is a Type~$\III$ analogue of 
Theorem~\ref{ThDeltaOne} and Theorem~\ref{ThDeltaTwo}.

\begin{thm}\label{thm:III}
	Let $C$ be a Type~$\III$ code over $\FF_3$ of length $n \equiv 0 \pmod 4$. 
	Then we have
	\begin{itemize}
		\item [(i)] 
		$\Delta_{\III}(C) 
		= 
		4-\dfrac{4}{3^{n/2-1}+1} \approx 4$,
		\item [(ii)] 
		$\Delta_{\III}^{2}(C) 
		= 
		\dfrac{40 (3^{n/2})^2}{(3^{n/2}+3)(3^{n/2}+9)} \approx 40$.
	\end{itemize}
\end{thm}

\begin{proof}
	(i) Let $C \in \III_n$. Then
	\begin{align*}
		\sum_{D \in \III_n} |C \cap D| 
		& = 
		\# \{(u,D) \in C \times \III_n \mid u \in D\} \\
		& = 
		\sum_{u \in C} 
		\#\{D\in \III_n \mid u \in D\} \\
		& = 
		\left( 
		\sum_{u=0} 
		+ 
		\sum_{u \in C \setminus \{0\}} 
		\right)
		\#\{D\in {\III}_n \mid u \in D\} \\
		& = 
		|{\III}_n| + (|C|-1) N_{n,1}^{\III}.
	\end{align*}
	Since 
	$|{\III}_n| = 2\prod_{i=1}^{n/2-1} (3^i + 1)$, 
	therefore we can write
	\begin{align*}
		\Delta_{\III}(C)
		& = 1 + (|C|-1) \dfrac{N_{n,1}^{\III}}{|{\III}_n|} \\
		& = 1 + \dfrac{3^{n/2}-1}{3^{n/2-1}+1} \\
		& = \dfrac{3^{n/2-1}+3^{n/2}}{3^{n/2-1}+1} \\
		& = \dfrac{3^{n/2-1}+3.3^{n/2-1}}{3^{n/2-1}+1} \\
		& = \dfrac{4.3^{n/2-1}}{3^{n/2-1}+1} \\
		& = 4 - \dfrac{4}{3^{n/2-1}+1}.
	\end{align*}
	This completes the proof of (i).
	
	(ii) Similarly as (i) we can write
	\begin{align*}
		\sum_{D \in \III_n} {|C \cap D|}^2 
		& = 
		\# \{(u,v,D) \in C \times C \times \III_n \mid u,v \in D\} \\
		& = 
		\sum_{u,v \in C} 
		\#\{D\in \III_n \mid \langle u,v \rangle \subseteq D\} \\
		& = 
		\left(
		\sum_{u,v=0} 
		+ 
		\sum_{\dim\langle u,v \rangle = 1} 
		+ 
		\sum_{\dim\langle u,v \rangle = 2}  
		\right)
		\#\{D\in \III_n \mid \langle u,v \rangle \subseteq D\} \\
		& = 
		|{\III}_n| 
		+ 
		4(|C|-1) N_{n,1}^{\III} 
		+ 
		(|C|-1)(|C|-3) N_{n,2}^{\III}.
	\end{align*}
	Since 
	$|{\III}_n| = 2\prod_{i=1}^{n/2-1} (3^i + 1)$, 
	therefore we can write
	\begin{align*}
	\Delta_{\III}^2(C)
		& = 
		1 
		+ 
		4(|C|-1) 
		\dfrac{N_{n,1}^{\III}}{|{\III}_n|} 
		+ 
		(|C|-1)(|C|-3) 
		\dfrac{N_{n,2}^{\III}}{|{\III}_n|} \\
		& = 
		1 
		+ 
		\dfrac{4(3^{n/2}-1)}{3^{n/2-1}+1} 
		+ 
		\dfrac{(3^{n/2}-1)(3^{n/2}-3)}{(3^{n/2-2}+1)(3^{n/2-1}+1)} \\
		& = 
		1 
		+ 
		\dfrac{12(3^{n/2}-1)}{3^{n/2}+3} 
		+ 
		\dfrac{27(3^{n/2}-1)(3^{n/2}-3)}{(3^{n/2}+9)(3^{n/2}+3)} \\
		& = 
		\dfrac{40 (3^{n/2})^2}{(3^{n/2}+3)(3^{n/2}+9)}.
	\end{align*}
	This completes the proof of (ii).
\end{proof}

We close this paper with
the following Type~$\IV$ analogue of 
Theorem~\ref{ThDeltaOne} and Theorem~\ref{ThDeltaTwo}.

\begin{thm}\label{thm:IV}
	Let $C$ be a Type~$\IV$ code over $\FF_4$ of length $n \equiv 0 \pmod 2$. 
	Then we have
	\begin{itemize}
		\item [(i)]  $\Delta_{\IV}(C) = 3 -\dfrac{3}{2^{2(n/2)-1}+1} \approx 3$,
		\item [(ii)]  $\Delta_{\IV}^{2}(C) = \dfrac{27(2^{2(n/2)})^2}{(2^{2(n/2)}+2)(2^{2(n/2)}+8)} \approx 27$.
	\end{itemize}
\end{thm}

\begin{proof}
	(i) Let $C \in \IV_n$. Then
	\begin{align*}
		\sum_{D \in \IV_n} |C \cap D| 
		& = 
		\# \{(u,D) \in C \times \IV_n \mid u \in D\} \\
		& = 
		\sum_{u \in C} 
		\#\{D\in \IV_n \mid u \in D\} \\
		& = 
		\left(
		\sum_{u=0} 
		+ 
		\sum_{u \in C\setminus \{0\}}
		\right)
		\#\{D\in \IV_n \mid u \in D\} \\
		& = 
		|\IV_n| + (|C|-1) N_{n,1}^{\IV}.
	\end{align*}
	Since 
	$|{\IV}_n| = \prod_{i=0}^{n/2-1} (2^{2i+1}+1)$, 
	therefore,
	\begin{align*}
		\Delta_{\IV}(C)
		& = 
		1 
		+ 
		(|C|-1) 
		\dfrac{N_{n,1}^{\IV}}{|\IV_n|} \\
		& = 
		1 
		+ 
		\dfrac{2^{2(n/2)}-1}{2^{2(n/2)-1}+1} \\
		& = 
		\dfrac{2^{2(n/2)-1}+2^{2(n/2)}}{2^{2(n/2)-1}+1} \\
		& = 
		\dfrac{3.2^{2(n/2)-1}}{2^{2(n/2)-1}+1} \\
		& = 
		3 - \dfrac{3}{2^{2(n/2)-1}+1}.
	\end{align*}
	This completes the proof of (i).
	 
	(ii) Similarly as (i) we can write
	\begin{align*}
		\sum_{D \in \IV_n} {|C \cap D|}^2 
		& = 
		\# \{(u,v,D) \in C \times C \times \IV_n \mid u,v \in D\} \\
		& = 
		\sum_{u,v \in C} 
		\#\{D\in \IV_n \mid \langle u,v \rangle \subseteq D\} \\
		& = 
		\left(
		\sum_{u,v=0} 
		+ 
		\sum_{\dim\langle u,v \rangle = 1} 
		+ 
		\sum_{\dim\langle u,v \rangle = 2}  
		\right)
		\#\{D\in \IV_n \mid \langle u,v \rangle \subseteq D\} \\
		& = 
		|{\IV}_n| 
		+ 
		5(|C|-1) N_{n,1}^{\IV} 
		+ 
		(|C|-1)(|C|-4) N_{n,2}^{\IV}.
	\end{align*}
	Since 
	$|{\IV}_n| = \prod_{i=0}^{n/2-1} (2^{2i+1}+1)$, 
	therefore,
	\begin{align*}
		\Delta_{\IV}^{2}(C)
		& = 
		1 
		+ 
		5(|C|-1)
		\dfrac{N_{n,1}^{\IV}}{|\IV_n|} 
		+ 
		(|C|-1)(|C|-4)
		\dfrac{N_{n,2}^{\IV}}{|\IV_n|} \\
		& = 
		1 
		+ 
		5\dfrac{2^{2(n/2)}-1}{2^{2(n/2)-1}+1} 
		+ 
		\dfrac{(2^{2(n/2)}-1)(2^{2(n/2)}-4)}{(2^{2(n/2)-3}+1)(2^{2(n/2)-1}+1)} \\
		& = 
		1 
		+ 
		10\dfrac{2^{2(n/2)}-1}{2^{2(n/2)}+2} 
		+ 
		16\dfrac{(2^{2(n/2)}-1)(2^{2(n/2)}-4)}{(2^{2(n/2)}+8)(2^{2(n/2)}+2)} \\
		& = 
		\dfrac{27{(2^{2(n/2)})}^2}{(2^{2(n/2)}+8)(2^{2(n/2)}+2)}.
	\end{align*}
	This completes the proof of (ii).
\end{proof}

\section*{Acknowledgements}
The authors thank Manabu Oura for helpful discussions. 
The authors would also like to thank the 
anonymous reviewers for their beneficial comments 
on an earlier version of the manuscript.
The second named author is supported by JSPS KAKENHI (18K03217).

\end{document}